\documentclass[12pt]{article}
\usepackage[utf8]{inputenc}
\usepackage{enumitem}
\usepackage{hyperref}
\usepackage{amsmath,amssymb, amsthm}
\usepackage{tgpagella}
\usepackage{euler}

\newtheorem{trm}{Theorem}[]
\newtheorem{lemma}[trm]{Lemma}
\theoremstyle{definition}

\theoremstyle{definition}

\theoremstyle{definition}

\theoremstyle{remark}

\newcommand{\normal}{\trianglelefteq}

\newcommand{\com}{\mathcal{C}}

\newcommand{\0}{\{0\}}

\newcommand{\N}{\mathbb{N}}

\newcommand{\auto}[1]{\text{Aut}\left( #1 \right)}

\usepackage{sectsty}
\allsectionsfont{\centering}

\author{James Bryden, Peter Rowley}
\date{\today}
\title{Automorphism Groups of the $PSL_2(q)$ Commuting Involution Graphs }

\begin{document}
\maketitle

\begin{abstract}
  Given a finite group $G$ and a conjugacy class of involutions $X$ of $G$,
  we define the commuting involution graph $\mathcal{C}(G,X)$ to be
  the graph with vertex set $X$ and $x,y \in X$ adjacent if and only if $x
  \neq y$ and $xy =yx$.
  In this paper the automorphism group of the graph $\mathcal{C}(G,X)$  is determined when $G = PSL_2(q)$.
\end{abstract}

\section{Introduction}\label{Introduction}
Given a finite group $G$ and a conjugacy class of involutions $X$ of $G$, the commuting involution graph on $X$, denoted $\com(G,X)$,
is the graph with vertex set $X$ in which distinct $x,y \in X$ are  adjacent if and only if $x$ and $y$ commute.
In \cite{FischerBernd1971Fggb}, Fischer constructed his three sporadic simple groups using the automorphism groups of commuting involution graphs. 
Since then, the commuting involution graphs for other finite simple groups have been investigated.
One can consult  \cite{Peter1, MR3695780} for symmetric and alternating groups, \cite{Bates1, MR2861687,Wright} for the sporadic groups, and \cite{MR4280329, comgraphspeciallinear, Ellis, MR2811072, 4dimsimp} and \cite{smallree} for groups of Lie-type.
The focus of much of this research is to describe the connectivity, diameters and disc structures of these graphs. 
Besides the work of Fischer and a result from \cite{MR3695780}, almost nothing is known about the automorphism groups of commuting involution graphs of simple groups.
In this paper, we take a small step towards filling that void. 
\begin{trm}\label{trm:1}
  Suppose that  $q$ is a prime power and $L = PSL_2(q)$.  Let $X$ denote the conjugacy class of involutions in $L$ and $A$ the automorphism group of the graph $\mathcal{C}(L,X)$. Then the following hold.
\begin{enumerate}
\item[(i)] If $q$ is even, then $A \cong \mathrm{Sym}(q - 1) \wr \mathrm{Sym}(q + 1)$.
\item[(ii)]  If $q$ is odd and $q \geq 7$,  then $A \cong P\Gamma L_2(q)$.
\end{enumerate}  
\end{trm}

In the case when $q$ is even, $\mathcal{C}(L,X)$ has $q + 1$ connected components each of which is a complete graph on $q - 1$ vertices.
This observation immediately gives part (i) of Theorem~\ref{trm:1}, and so this paper is devoted to proving part (ii).
We begin in Section \ref{Preliminary} assembling properties of $\mathcal{C}(L,X)$ which we need for our proof.
There are two results which we single out on account of their playing an important role, namely Lemma~\ref{lem:factorisation}  and Theorem~\ref{Solutions}.
Lemma~\ref{lem:factorisation} indicates that vertex stabilizers in $A$ hold the key to determining $A$, while Theorem~\ref{Solutions} features prominently in Lemmas~\ref{lem:bound1} and ~\ref{lem:bound2}.
The threads proving Theorem~\ref{trm:1}(ii) are drawn together in Section~\ref{sec:main}, where we have to subdivide into the two cases $q \equiv 1 \mod{4}$ and $q \equiv 3 \mod{4}$.
This is necessary as, depending on the congruence of $q$, $\mathcal{C}(L,X)$ displays a number of different graph theoretic properties.

\section*{Acknowledgments}
We would like to thank Philip Dittmann for his help and advice relating to Theorem~\ref{Solutions}. 
The first author's  work was supported by the Heilbronn Institute for Mathematical Research.

\section{Preliminary Results}\label{Preliminary}

Throughout the remainder of this paper, $q \geq 7$ denotes an odd prime power, $L = PSL_2(q)$, $X$ is the conjugacy class of involutions in $L$, and we set $G = PGL_2(q)$.
Also set  $A = \mathrm{Aut}(\mathcal{C}(L,X))$.
We need to work explicitly with the elements of $X$, and so we use representatives in $SL_2(q)$.
Due to this choice, we  have the peculiarity that when $x,y \in X$ are distinct they commute in $L$ if and only if $xy=-yx$ in $SL_2(q)$.

For $x,y \in X$ we use the notation $d(x,y)$ to be the  graph distance metric on $\mathcal{C}(L,X)$.
That is, $d(x,y)$ is the length of a shortest path from $x$ to $y$ in $\mathcal{C}(L,X)$ and $\infty$ if no such path exists.
By the diameter of $\mathcal{C}(L,X)$ we mean the quantity $\max\,\{ d(x,y) \mid x,y \in X\}$.
Additionally, for a vertex $x \in X$ we define the $i$th disc (about $x$) to be the set
$\Delta_i(x) = \{ y \in X \mid d(x,y) = i\}$.
Finally, $A_x$ denotes the stabilizer in $A$ of $x$.

\subsection{The Structure of $\mathcal{C}(L,X)$}
As is to be expected, many structural results for $\mathcal{C}(L,X)$ can be obtained from basic  properties of the involution centralizers in  $L$.
Whilst some of these observations are agnostic of which field we are working over, much of the structure of $L$ and so $\mathcal{C}(L,X)$ are  very sensitive to changes in $q$. 
For our purposes, we divide our graphs into the two cases $q \equiv 1 \mod{4}$ and $q \equiv 3\mod{4}$.
However, it is worth noting that even  within these two classes the graphs exhibit quite different structures.

Before examining more graph-theoretic aspects of $\mathcal{C}(L,X)$, we discuss the  visible symmetries $\mathcal{C}(L,X)$ possesses. 
Notice that, as $L$ acts transitively on $X$ we have that $\mathcal{C}(L,X)$ is a vertex transitive graph. 
Moreover, as $L = \langle X \rangle$ and $X$ is the unique involution class in $L$, we have that $\auto{L} \cong P\Gamma L_2(q)$ acts faithfully on $\mathcal{C}(L,X)$.
It is then the role of Theorem~\ref{trm:1}(ii) to show that $A$ is no bigger than $\auto{L}$. 
A useful observation to make here is that we get a nice factorization of $A$. 
\begin{lemma}\label{lem:factorisation}
  For any vertex $x \in X$ we have $A = A_x L$. 
\end{lemma}
\begin{proof} This follows from $L$ being transitive on $X$.
\end{proof}

For a vertex $x \in X$ its stabilizer in $L$ is clearly $C_L(x)$ which is a dihedral group. 
The neighbours of $x$ in $\mathcal{C}(L,X)$ form the first disc, $\Delta_1(x) \subseteq C_L(x)$, and $L_x = A_x \cap L$ partitions  $\Delta_1(x)$ into two orbits. 
However, $G_x = A_x \cap G$ - also a dihedral group - acts transitively on $\Delta_1(x)$.
It is for this reason that we sometimes prefer to use $G_x$ instead of $L_x$ in our arguments as a  known subgroup of $A_x$.

The involution structure in  $L$ forces vertex neighbourhoods in $\com(L,X)$ to have very small intersections.
As we tackle small values of $q$ computationally, we assume $q > 11$ in the next two lemmas,  and hence $L_x$ is a maximal subgroup of $L$.
However, we note that Lemmas~\ref{lem:4cycle} and \ref{lem:eigenvalues} are true for all odd $q$.
\begin{lemma}\label{lem:4cycle}
  Assume $q > 11$. 
  Then for all distinct $x,y \in X$ we have $|\Delta_1(x) \cap \Delta_1(y)| \leq 1$.
  In particular, the graph $\mathcal{C}(L,X)$ is $4$-cycle free. 
\end{lemma}
\begin{proof} 
  Let $x,y \in X$ be distinct involutions.
  Then $C_L(x)$ and $C_L(y)$ are maximal subgroups of $L$ which are dihedral groups.
  Suppose that $g \in C_L(x) \cap C_L(y)$ is such that $g$ has order at least $3$.
  Then $N_L(\langle g \rangle ) \geq C_L(x), C_L(y)$ and thus $N_L(\langle g \rangle) = L$.
  This contradicts the fact that $L$ is simple.
  Therefore $C_L(x) \cap C_L(x)$ is either trivial or an elementary  abelian 2-group of order 2 or 4.
  In the case that $C_L(x) \cap C_L(y)$ is a fours-group, we must have $C_L(x) \cap C_L(y) = \{1,x,y,xy\}$.
  In any case, $|\Delta_1(x) \cap \Delta_1(y)| \leq 1$.
  In particular, $\mathcal{C}(L,X)$ contains no subgraphs which are 4-cycles. 
\end{proof}

The next lemma allows us to test whether two vertices are close together in $\com(L,X)$.
\begin{lemma}\label{lem:eigenvalues}
  Assume $q > 11$ and let $x,y \in X$ be distinct involutions. 
  If $q \equiv 1\mod{4}$, then $d(x,y) \leq 2$ if and only if $xy$ has two distinct eigenvalues, and 
  if $q \equiv 3\mod{4}$, then $d(x,y) \leq 2$ if and only if $xy$ does not have two distinct eigenvalues. 
\end{lemma}

  
\begin{proof} 
  The group $G$ has two conjugacy classes of involutions and they are characterised by how many points of the projective line their elements fix.
  That is, one conjugacy class contains involutions fixing exactly two points of the projective line, and the other contains involutions which fix no points of the projective line.
  Suppose that $x,y \in G$ are involutions such that $xy$ has two distinct eigenvalues.
  Then $xy$ fixes two points of the projective line, $P$ and $Q$.
  Moreover, $G_{\{P,Q\}} = C_G(z)$ for some involution $z$ fixing $P$ and $Q$. 
  As $C_G(z)$ is a dihedral group, and is maximal in $G$, $xy \in C_G(z)$ implies that $x,y \in C_G(z)$. 
  Thus $xy$ has two distinct eigenvalues if and only if $x,y \in C_G(z)$ for some involution $z$ fixing two projective line points. 
  
  In the case that  $q \equiv 1 \mod{4}$ the unique involution class of $L \leq G$ contains involutions fixing two points of the projective line.
  So when $x,y \in L$ are involutions it follows that  $xy$ has two distinct eigenvalues if and only if  $x,y \in C_L(z)$ for some involution $z \in L$, if and only if $d(x,y)  \leq 2$.
  In the case that $q \equiv 3\mod{4}$ the unique involution class of $L$ contains elements fixing no points of the projective line. As such, involutions $x,y \in L$ satisfy $xy$ has two distinct eigenvalues if and only if $d(x,y) >  2$.
\end{proof}

In  \cite{comgraphspeciallinear}, it is shown that when $q$ is not too small the graph $\mathcal{C}(L,X)$ has diameter 3.
In addition, when $q \equiv 1 \mod{4}$ the discs are determined explicitly.
The next lemma collates this information and extends the explicit description of the discs  to the case $q \equiv 3\mod{4}$.
However, as $\com(L,X)$ is vertex transitive, we have the freedom to choose any vertex to centre our discussion on.
When $q \equiv 1\mod{4}$, we set 
\[
  t =  \begin{bmatrix}    \iota &0 \\ 0 & -\iota   \end{bmatrix} \in X
 \]
 where $\iota \in GF(q)$ is such that $\iota^2 = -1$, and when $q\equiv3\mod{4}$  
 \[
  t =
  \begin{bmatrix}
    0 & 1 \\ -1 & 0
  \end{bmatrix} \in X.
\] 
\begin{lemma}\label{lem:disks}
  When $q \equiv 1\mod{4}$ and $q \geq 17$  or $q\equiv3\mod{4}$ and $q \geq 7$ the graph $\mathcal{C}(L,X)$ has diameter 3.
  Moreover,  if $q\equiv 1\mod{4}$ the discs of $t$ are as follows.
  \begin{enumerate}
  \item[(i)] $\Delta_1(t) = \left\{ \begin{bmatrix} 0 & \omega\\ -1/\omega & 0 \end{bmatrix} \,\middle\vert\, \omega \neq 0 \right\}$.
  \item[(ii)] $\Delta_2(t) = \left\{ \begin{bmatrix} \sigma & \mu^2 \tau\\\tau & -\sigma \end{bmatrix}\,\middle\vert\,\begin{matrix} \sigma, \mu, \tau \neq 0\\ \sigma^2+\mu^2\tau^2 = -1\end{matrix}\right\}$.
  \item[(iii)] $\Delta_3(t) = \left\{ \begin{bmatrix} \iota & \alpha \\0&-\iota\end{bmatrix},\,\begin{bmatrix}\iota&0\\\alpha&-\iota\end{bmatrix} \,\middle\vert\, \alpha \neq 0 \right\} \cup \left\{ \begin{bmatrix} \sigma & \mu \tau\\\tau & -\sigma \end{bmatrix}\,\middle\vert\, \begin{matrix} \sigma, \mu, \tau \neq 0\\ \mu \text{ is non-square}\\ \sigma^2+\mu\tau^2 = -1\end{matrix} \right\}$.  
  \end{enumerate}
  If $q \equiv 3\mod{4}$  the discs of $t$ are as follows.
\begin{enumerate}
  \item[(i)] $\Delta_1(t) =
    \left\{
      \begin{bmatrix}
        \sigma & \tau \\
        \tau & - \sigma 
      \end{bmatrix}
      \, \middle\vert \,
      \sigma^2 + \tau^2 = -1.
    \right\}$.
  \item[(ii)] $\Delta_2(t) = 
    \left\{
      \begin{bmatrix}
        -\frac{(\beta+ \gamma)\tau}{2\sigma} & \beta \\
        \gamma & \frac{(\beta+ \gamma)\tau}{2\sigma}
      \end{bmatrix}
      \,\middle\vert\,
      \begin{matrix}
        \sigma^2+\tau^2 = -1,\\
        \frac{(\beta+ \gamma)^2\tau^2}{4\sigma^2} + \beta\gamma = -1\\
        \beta \neq \gamma
      \end{matrix}
    \right\}$.
  \item[(iii)] $\Delta_3(t) =
    \left\{
      \frac{1}{1+\alpha^2}
      \begin{bmatrix}
        \alpha(\omega^{-1}  - \omega) &  \alpha^2\omega^{-1}  + \omega \\ 
        -\alpha^2\omega - \omega^{-1}  &  -\alpha(\omega^{-1} - \omega) 
      \end{bmatrix}
      \,\middle\vert\,
      \begin{matrix}
        \alpha \\
        \omega \neq 0,\pm 1
        \end{matrix}
    \right\}$.
  \end{enumerate}
\end{lemma}
\begin{proof}
  In the case that $q \equiv 1 \mod{4}$ and $q \geq 17$ the diameter and discs  are described  by Proposition~2.4, Proposition~2.5 and  Theorem~2.7  in \cite{comgraphspeciallinear}.

  Suppose  then that $q \equiv 3 \mod{4}$ and $q \geq 7$.
  Let $x \in X$ be given by $\begin{bmatrix} \alpha& \beta \\ \gamma&-\alpha\end{bmatrix}$.
  If $x$ and $t$ commute in $L$ then, in $SL_2(q)$, we have $xt=-tx$. This yields the relation $\beta=\gamma$. 
  Thus,
  \[
    \Delta_1(t) = \left\{ \begin{bmatrix} \sigma & \tau \\ \tau & -\sigma \end{bmatrix}\,\middle\vert\, \sigma^2  + \tau^2 = -1\right\}.
  \]
  Fix $s = \begin{bmatrix}\sigma & \tau \\ \tau &-\sigma \end{bmatrix}\in\Delta_1(t)$ and suppose that $x \in \Delta_1(s)$.
  Then, $\alpha = -\frac{(\beta +\gamma)\tau}{2\sigma}$. If $\beta =\gamma$ then $x \in \Delta_1(t)$. 
  Thus
  \[
    \Delta_1(s) \cap \Delta_2(t) =
    \left\{ \begin{bmatrix} -\frac{(\beta +\gamma)\tau}{2\sigma} & \beta \\ \gamma & \frac{(\beta +\gamma)\tau}{2\sigma}\end{bmatrix}\,
      \middle\vert
      \,
      \begin{matrix}
        \frac{(\beta +\gamma)^2\tau^2}{4\sigma^2} + \beta\gamma = -1\\
        \beta \neq \gamma
      \end{matrix}
    \right\},
  \]
  and so
  \[
    \Delta_2(t) =
    \left\{ \begin{bmatrix} -\frac{(\beta +\gamma)\tau}{2\sigma} & \beta \\ \gamma & \frac{(\beta +\gamma)\tau}{2\sigma}\end{bmatrix}\,
      \middle\vert
      \,
      \begin{matrix}
        \sigma^2+\tau^2 = -1 \\
        \frac{(\beta +\gamma)^2\tau^2}{4\sigma^2} + \beta\gamma = -1\\
        \beta \neq \gamma
      \end{matrix}
    \right\}.
  \]

  Now, by Theorem~1.1 in \cite{comgraphspeciallinear}, every other vertex unaccounted for must lie in $\Delta_3(t)$.
  We determine $\Delta_3(t)$ by determining $G_t = C_G(t)$ and its orbits on $\Delta_3(t)$.

  Suppose that $g = \begin{bmatrix}\alpha & \beta\\ \gamma & \delta\end{bmatrix}$ is such that $g \in C_G(t)$. 
  Then, $gt = \epsilon tg$ for some $\epsilon \in GF(q)^*$.
  But, (in $GL_2(q)$) we have  $t^{-1} = -t$.
  Thus, as $g^{t^2} = g$ we must have $\epsilon^2 = 1$ and so  $\epsilon = \pm 1$. 

  If $gt = tg$, then $\delta = \alpha$ and $\gamma = -\beta$. 
  If $gt = -tg$, then $\delta= -\alpha$ and $\gamma = \beta$.
  Hence $g$ takes one of the two forms
  \[
    \begin{bmatrix}
      \alpha & \beta \\
      -\beta & \alpha 
    \end{bmatrix}
    \text{ or }
    \begin{bmatrix}
      \alpha & \beta \\
      \beta & -\alpha 
    \end{bmatrix}.
  \]
  However,  we can ignore scalars in $G$ and so $g$ can be  represented by one of
  \[
    g_{\alpha}:=
    \begin{bmatrix}
      1 & \alpha \\
      -\alpha & 1 
    \end{bmatrix}
    \text{ or }
    h_{\alpha}:=
    \begin{bmatrix}
      \alpha & 1 \\
      1 & -\alpha 
    \end{bmatrix}.
  \]
  Observe that,
  \[
    g_{\alpha}^{-1} = \frac{1}{1+\alpha^2}g_{-\alpha}  \text{ and }
    h_{\alpha}^{-1} = \frac{1}{1+\alpha^2}h_{\alpha}.
  \]

  Letting $t_\omega = \begin{bmatrix}0&\omega\\-1/\omega&0\end{bmatrix}$, we have that $t = t_1$.
  Moreover, $t_\omega \in \Delta_3(t)$ for all $\omega \neq 0,\pm 1$.
  Lastly, $t_\omega$ and $t_\lambda$ are in the same $G_t$-orbit if and only if $\lambda \in \{ \pm \omega, \pm 1/\omega\}$.
  Hence, the $t_\omega$ represent the $G_t$-orbits of $\Delta_3(t)$. 

  We have
  \[
    t_\omega^{g_{\alpha}} =\frac{1}{1+\alpha^2}g_{-\alpha} t_\omega g_{\alpha}
    =\frac{1}{1+\alpha^2}
    \begin{bmatrix}
      \alpha(\omega^{-1}  - \omega) &  \alpha^2\omega^{-1}  + \omega \\ 
      -\alpha^2\omega - \omega^{-1}  &  -\alpha(\omega^{-1} - \omega) 
    \end{bmatrix}
  \]
  and
  \[
    t_\omega^{h_{\alpha}} = \frac{1}{1+\alpha^2}h_{\alpha} t_\omega h_{\alpha} =
    \frac{1}{1+\alpha^2}
    \begin{bmatrix}
      \alpha(\omega  - \omega^{-1})&  -\alpha^2\omega  - \omega^{-1} \\
      \alpha^2\omega^{-1}  + \omega  &  -\alpha(\omega  - \omega^{-1})
    \end{bmatrix}.
  \]

  Observe that, if $\alpha \neq 0$ then $t_\omega^{g_{\alpha}} =- t_\omega^{h_{1/\alpha}}$,
  whereby, 
  \[
    t_\omega^{G_t} =
    \left\{
      \begin{bmatrix}
        0 & -\omega^{-1} \\ \omega & 0 
      \end{bmatrix}\,,\,
      \frac{1}{1+\alpha^2}
      \begin{bmatrix}
        \alpha(\omega^{-1}  - \omega) &  \alpha^2\omega^{-1}  + \omega \\ 
        -\alpha^2\omega - \omega^{-1}  &  -\alpha(\omega^{-1} - \omega) 
      \end{bmatrix}
      \,\middle\vert\,
      \alpha 
    \right\}.
  \]
\end{proof}

\subsection{Solutions of Certain Polynomials}

Recalling that a polynomial over $GF(q)$ is said to be absolutely irreducible if it is irreducible as a polynomial over the algebraic closure of $GF(q)$, we have the following appearing as Theorem 6.57 in \cite{finitefields} which gives bounds on sizes of certain solution sets.

\begin{trm}\label{Solutions}
  Let $m \in \N$  and let $f \in GF(q)[x]$  with $\deg{f} >1$ be 
  such that $y^k -f(x)$ is absolutely irreducible, where $k = gcd(m, q - 1)$.
  Then the number $N$ of solutions of the equation $y^m = f(x)$ 
  in $GF(q)^2$ satisfies $|N-q| \leq (k-1)(d-1)\sqrt{q}$
  where $d$ is the number of distinct roots of $f$ in its splitting field over $GF(q)$.
\end{trm}

\section{The Automorphism Group of $\mathcal{C}(L,X)$}\label{sec:main}
In this section we prove Theorem~\ref{trm:1}(ii). As mentioned in Section~\ref{Introduction}, the proof breaks into two cases $q \equiv  1\mod{4}$ and $q \equiv 3\mod{4}$. However, in either case the overall strategy is similar. For $t$ defined in Section~\ref{Preliminary}, let $K$ be the normal subgroup of $A_t$ fixing $\Delta_1(t)$ vertex-wise. Then we investigate the orbits of $K$ on $\Delta_2(t)$, when $q \equiv  1\mod{4}$, and the orbits of $K$ on $\Delta_3(t)$ when $q \equiv 3\mod{4}$.  With this information we are then able to show $t \in \mathrm{Z}(A_t)$, provided 
$q \geq 73$ (for $q \equiv  1\mod{4}$) and $q \geq 67$ (for $q \equiv  3\mod{4}$). Knowing that  $t \in \mathrm{Z}(A_t)$ swiftly leads to pinning down $A$, and the small values of $q$ may be quickly checked using \textsc{Magma}~\cite{magma}.

\subsection{The Case $q \equiv 1 \mod{4}$} 
Since $q \equiv 1 \mod{4}$ we may choose $\iota \in GF(q)$ such that $\iota^2 =-1$.
Along with our involution $t$ we fix $s \in \Delta_1(t)$, so as 
\[
  t =  \begin{bmatrix}    \iota &0 \\ 0 & -\iota   \end{bmatrix}
  \text{ and }
  s = 
  \begin{bmatrix}
    0 & 1 \\
    -1 & 0 
  \end{bmatrix}.
\]
Observe that, by Lemma~\ref{lem:4cycle} we have that every element of $\Delta_2(t)$ is adjacent to a unique   vertex in $\Delta_1(t)$.
Moreover, $A_t \geq G_t$ is transitive on $\Delta_1(t)$.
Hence,  to show that the $K$-orbits of $\Delta_2(t)$ take the form $\{v,v^t\}$ it is sufficient to  show that the $K$-orbits of $\Delta_1(s) \cap \Delta_2(t)$ take this form.     

We first describe  $\Delta_1(s) \cap \Delta_2(t)$ explicitly.
\begin{lemma}
 We have that
  \[
    \Delta_1(s) \cap \Delta_2(t) =
    \left\{
      \begin{bmatrix}
        \sigma & \tau \\ \tau & -\sigma 
      \end{bmatrix}\, \middle\vert \,
      \begin{matrix}
        \sigma, \tau \neq 0 \\
        \sigma^2 + \tau^2  = -1 
      \end{matrix}
    \right\}.
  \]
\end{lemma}
\begin{proof}
  Let $x = \begin{bmatrix} \alpha & \beta \\ \gamma & -\alpha \end{bmatrix} \in X$.
  Then $x \in \Delta_1(s)$ if and only if $xs = -sx$ and $x \neq \pm s$.
  The equation $xs+sx=0$ yields $\beta = \gamma$.
  So $x = \begin{bmatrix} \sigma & \tau \\ \tau & -\sigma \end{bmatrix}$ with $\sigma^2 + \tau^2 = -1$.
  However, if $\sigma$ or $\tau =0$ clearly we have $x \in \Delta_1(t)$ or $x= \pm t$ respectively.
  Therefore
  \[
    \Delta_1(s) \cap \Delta_2(t) =
    \left\{
      \begin{bmatrix}
        \sigma & \tau \\ \tau & -\sigma 
      \end{bmatrix}\, \middle\vert \,
      \begin{matrix}
        \sigma, \tau \neq 0 \\
        \sigma^2 + \tau^2  = -1 
      \end{matrix}
    \right\},
  \] as desired.
\end{proof}

Define $s_{\sigma, \tau} = \begin{bmatrix}\sigma & \tau \\\tau&-\sigma\end{bmatrix} \in \Delta_1(s)\cap\Delta_2(t)$ and further set
$E_{\sigma,\tau} = \Delta_2(s_{\sigma,\tau}) \cap \Delta_1(t)$.
Moreover, notice that $s_{\sigma,\tau} = s_{-\sigma,-\tau}$ and $s_{\sigma,\tau}^t = s_{\sigma, -\tau}$.

The next two lemmas form the heart of our argument.   
\begin{lemma}\label{lem:poly1}
  Suppose that $q > 11$ and $t_\omega = \begin{bmatrix} 0 & \omega \\ -1/\omega & 0\end{bmatrix} \in \Delta_1(t)$.
  Then, $t_\omega \in E_{\sigma,\tau}$ if and only if $\omega^4 -(2+4/\tau^2)\omega^2 + 1$ is a non-zero square in $GF(q)$.
\end{lemma}
\begin{proof}
  By Lemma~\ref{lem:eigenvalues}, we have that $t_\omega \in E_{\sigma,\tau}$ if and only if $t_\omega s_{\sigma,\tau}$ has two distinct eigenvalues.
  Observe that the characteristic polynomial of $t_\omega s_{\sigma,\tau}$ is given by
  \[
    \chi(x) =
    1 + (\omega^{-1} -\omega)\tau x  + x^2 
  \]  which has discriminant
  \[
    \phi(\omega) = \tau^2(\omega^{-1}-\omega)^2 -4.
  \]
  Now  $\phi(\omega)$ is a non-zero square in $GF(q)$ if and only if
  \[
    \frac{\omega^2}{\tau^2} \phi(\omega)
    =
    \omega^4 -(2+4/\tau^2)\omega^2 + 1
  \] is a non-zero square in $GF(q)$.
  This proves the lemma.
\end{proof}

\begin{lemma}\label{lem:bound1}
 Assume that $q \geq 73$. Then $E_{\rho,\mu} = E_{\sigma,\tau}$ if and only if $(\rho,\mu) = (\pm\sigma, \pm \tau)$.
\end{lemma}
\begin{proof}
  Fix $s_{\sigma,\tau}$ and $s_{\rho,\mu}$ in $\Delta_1(s) \cap \Delta_2(t)$.
  Notice that, as $s_{\sigma,\tau}^t = s_{\sigma,-\tau}$ we have that $E_{\sigma, \tau} = E_{\pm \sigma, \pm\tau}$.

  Assume, for a contradiction, that $(\rho,\mu) \neq (\pm \sigma, \pm \tau)$ and $E_{\rho,\mu}  = E_{\sigma, \tau}$.
  Define $\gamma = -2-4/\tau^2$, $\delta =-2-4/\mu^2$, $f_{\gamma}(\omega) = \omega^4 + \gamma\omega^2 + 1$, and $f_{\delta}(\omega) = \omega^4+ \delta\omega^2 + 1$.
  Then, by Lemma~\ref{lem:poly1} $f_{\gamma}(\omega)$ is a non-zero square if and only  if $f_{\delta}(\omega)$ is a non-zero square.
 Therefore, $f(\omega) := f_{\gamma}(\omega)f_{\delta}(\omega)$ is a (possibly zero) square for all $\omega \in GF(q)$.

  We now verify that $f$ is not the square of a polynomial in $\overline{GF(q)}[\omega]$ and thus show that the polynomial $y^2 -f(x) \in GF(q)[x,y]$ is absolutely irreducible.
  Observe that  any common factor of $f_{\gamma}$ and $f_{\delta}$ is a factor of $(f_{\gamma} - f_{\delta})(\omega)=(\gamma - \delta) \omega^2$, 
  and hence  $f_{\gamma}$ and $f_{\delta}$ are coprime.
  Moreover, as $\gamma,\delta \neq \pm 2$, we have that $f_{\gamma}$ and $f_{\delta}$ are  separable. 
  Thus, $f$ has no repeated linear factors and so cannot be a square in $\overline{GF(q)}[\omega]$.
  Hence, $y^2-f(x)\in GF(q)[x,y]$ is absolutely irreducible.

  Define $N = |\{ (x,y) \in GF(q)^2 \mid y^2= f(x) \}|$  and observe that, as $\deg{f} =8$, we have $N \geq 2q-8$.
Applying Theorem~\ref{Solutions} yields that $|N-q| \leq 7\sqrt{q}$ and hence 
                   $$q^2 -65q + 64=(q-1)(q-64) \leq 0,$$
                   whereby  $q \leq 64$.
                   However, by assumption $q \geq 73$  which gives a contradiction.
\end{proof}

We require one more lemma before we can prove the $q\equiv 1\mod{4}$ part of Theorem~\ref{trm:1}.
\begin{lemma}\label{lem:faithful1}
  If $q \geq 17$, then $A_t$ acts faithfully on $\Delta_2(t)$.
\end{lemma}
\begin{proof}
  Suppose that $g \in A_t$ fixes $\Delta_2(t)$ vertex-wise.
  By Proposition~2.4 in \cite{comgraphspeciallinear} each element of $\Delta_2(t) \cup \Delta_3(t)$ is adjacent to a unique element of the form
  $\begin{bmatrix} \iota & \alpha \\0&-\iota\end{bmatrix}$ for $\alpha \neq 0$.
  It follows that if $u = \begin{bmatrix} \iota & \alpha \\ 0&-\iota\end{bmatrix} \in \Delta_3(t)$ then $|\Delta_1(u) \cap \Delta_2(t)| =\frac{q-5}{4}$.
  Since $q \geq 17$, we have that $|\Delta_1(u)\cap\Delta_2(t)| \geq 2$.
  Then, by Lemma~\ref{lem:4cycle} we must have $u^g = u$.
  Suppose now that $v \in \Delta_1(u) \cap \Delta_3(t)$.
  Then, as $v \in \Delta_3(t)$, we have $v' \in \Delta_2(t) \cap \Delta_1(v)$.
  So $v$ and $v^g$ both have $u$ and $v'$ as neighbours.
  Thus, Lemma~\ref{lem:4cycle} yields that $v = v^g$.
  It follows that $g = 1$ and $A_t$ acts faithfully on $\Delta_2(t)$.
\end{proof}

We can now prove Theorem~\ref{trm:1}(ii) when $q$ is large enough.
\begin{lemma}\label{lem:1}
  If $q \geq 73$, then $A = P\Gamma L_2(q)$ holds.
\end{lemma}
\begin{proof}
  Suppose that $q \geq 73$ and that  $K \normal A_t$  is the vertex-wise stabilizer of $\Delta_1(t)$ in $A_t$.
  By Lemma~\ref{lem:bound1} we have that the $K$-orbits of $\Delta_1(s) \cap \Delta_2(t)$ all take the form $\{v, v^t\}$.
  As $A_t$ acts transitively on $\Delta_1(t)$, this implies that $K$ partitions  $\Delta_2(t)$ into orbits of the form $\{v,v^t\}$.
  Moreover, observe that $t \in K$ fixes no vertices of $\Delta_2(t)$.
  Suppose that $g \in K$ also enjoys this property so that $gt$ fixes $\Delta_2(t)$ vertex-wise.
  Hence, by Lemma~\ref{lem:faithful1}, we have $g = t$.
  That is, $t$ is the unique element of $K$ fixing no vertices of $\Delta_2(t)$ and hence $\langle t \rangle \normal A_t$,
  whereby $t^A = t^L$ by Lemma~\ref{lem:factorisation}.
  As $t^L$ generates $L$ it follows that $L \normal A$.
  Moreover, clearly $C_A(L) = 1$ as $A$ acts faithfully on the graph $\com(L,X)$.
  It then follows by the fact that $\auto{L} \cong P\Gamma L_2(q) \leq A$ that $A  = P\Gamma L_2(q)$.
  This proves the result. 
\end{proof}

\subsection{The Case $q \equiv 3 \mod{4}$}
Suppose that $q \equiv 3\mod{4}$ and let
\[
  t =
  \begin{bmatrix}
    0 & 1 \\ -1 & 0
  \end{bmatrix} \in X 
\] be our fixed involution.

Again, let $K \normal A_t$ be the subgroup of $A_t$ fixing $\Delta_1(t)$ vertex-wise.
We argue to show that the $K$-orbits of $\Delta_3(t)$ all take the form $\{v,v^t\}$.

Define $s_{\alpha,\omega} \in \Delta_3(t)$ to be the involution with parameters $\alpha$ and $\omega$ 
and define $E_{\alpha,\omega} = \Delta_2(s_{\alpha,\omega}) \cap \Delta_1(t)$.
Recall from the proof of Lemma~\ref{lem:disks} that every $s_{\alpha,\omega}$ is contained in  a $G_t$-orbit represented by $s_{0,\omega}$.
As $K$ is normal in $A_t$, it is therefore sufficient to show that the $K$-orbit containing $s_{0,\omega}$ takes the form $\{ s_{0,\omega}, s_{0,\omega}^t \}$.

\begin{lemma}\label{lem:poly2}
 Assume that $q > 11$. The vertex $\begin{bmatrix}\sigma&\tau\\\tau &-\sigma\end{bmatrix} \in \Delta_1(t)$ belongs to $E_{\alpha,\omega}$ if and only if
  \begin{align*}
    & 4\alpha(\alpha^2 - 1)(\omega^2 + \omega^{-2}-2)\sigma\tau \\
    &+ (\alpha^4-6\alpha^2 +1)(\omega^2+\omega^{-2} -2) \tau^2\\
    &-4(\alpha^4 + \alpha^2\omega^2 + \alpha^2\omega^{-2} +1)
  \end{align*}
  is not a non-zero square in $GF(q)$.
\end{lemma}
\begin{proof}
  By Lemma~\ref{lem:eigenvalues}, we have $z=\begin{bmatrix}\sigma&\tau\\\tau &-\sigma\end{bmatrix} \in \Delta_1(t)$ belongs to $E_{a,\omega}$ if and only if $zs_{a,\omega}$
  does not have two distinct eigenvalues.
  Moreover, multiplying by a non-zero scalar does not change the number of distinct eigenvalues a matrix has.
  Whereby, $zs_{\alpha,\omega}$ has two distinct eigenvalues if and only if $(1+\alpha^2)zs_{\alpha,\omega}$ has two distinct eigenvalues.
  The discriminant of the characteristic polynomial of $(1+\alpha^2)zs_{\alpha,\omega}$ is given by the stated polynomial. This proves the lemma.

\end{proof}

\begin{lemma}\label{lem:bound2}
  If $q \geq 67$, then $E_{\alpha,\omega}=E_{\alpha',\omega'}$ if and only if $s_{\alpha',\omega'} =s_{\alpha,\omega}$ or $s_{\alpha',\omega'} = s_{\alpha,\omega}^t$.
\end{lemma}
\begin{proof}
  As we have noted before, as every $s_{\alpha,\omega}$ in the same $G_t$-orbit as $s_{0,\omega}$, it is sufficient to show that
  $E_{\alpha,\omega} = E_{0,\lambda}$ if and only if $\alpha=0$ and $\omega =\pm \lambda, \pm \lambda^{-1}$.

  One direction of the implication is trivial.
  Thus, towards a contradiction,  assume  that $E_{\alpha,\omega} = E_{0,\lambda}$ but $(\alpha,\omega) \neq (0, \pm \lambda^{\pm 1})$.
  Then, by Lemma~\ref{lem:poly2} it holds that for all $(\sigma, \tau)$ with $\sigma^2+\tau^2 =-1$, 
  \begin{equation}\label{eq1}
    \begin{split}
    &4\alpha(\alpha^2 - 1)(\omega^2 + \omega^{-2}-2)\sigma\tau \\
    &+ (\alpha^4-6\alpha^2 +1)(\omega^2+\omega^{-2} -2) \tau^2\\
      &-4(\alpha^4 +\alpha^2\omega^2 + \alpha^2\omega^{-2} +1)
      \end{split}
  \end{equation}
  is a non-zero square in $GF(q)$ if and only if
  \begin{equation}\label{eq2}
    (\lambda^2+\lambda^{-2} -2) \tau^2 -4
  \end{equation}
  is a non-zero square in $GF(q)$.
  
  It is possible to parameterize all but at most one point on the curve $\sigma^2 + \tau^2 = -1$ in terms of a single variable $\zeta$ over $GF(q)$.
  Fix a point $(\sigma_0,\tau_0) \in GF(q)^2$ such that $\sigma_0^2 + \tau_0^2 = -1$.
  Then, given any $\zeta \in GF(q)$, we recover a unique point $(\sigma,\tau)$ on the curve by taking 
  \begin{equation*}\label{eq3}
    \tau  =   \frac{ -\tau_0 - 2\sigma_0\zeta + \tau_0\zeta^2}{\zeta^2+1}
    \text{ and }
    \sigma = \frac{ \sigma_0  - 2\tau_0\zeta  - \sigma_0\zeta^2}{\zeta^2+1}.
  \end{equation*}
  Substituting for $\sigma$ and $\tau$ in (\ref{eq1}) and (\ref{eq2})  and  multiplying by $(\zeta^2 + 1)^2$ we obtain the two polynomials
  \begin{align*}
    \theta_{\alpha,\omega}(\zeta) &=  -4\alpha(\alpha^2 - 1)(\omega^2 + \omega^{-2}-2)\left( \sigma_0  - 2\tau_0\zeta  - \sigma_0\zeta^2\right)\left(\tau_0 + 2\sigma_0\zeta - \tau_0\zeta^2\right) \\
                             &+ (\alpha^4-6\alpha^2 +1)(\omega^2+\omega^{-2} -2) \left(\tau_0 + 2\sigma_0\zeta - \tau_0\zeta^2 \right)^2\\
                             &-4(\alpha^4 +\alpha^2\omega^2 + \alpha^2\omega^{-2} +1)(\zeta^2+1)^2
  \end{align*}
  and
  \[
    \theta_{0,\lambda}(\zeta) = 
    (\lambda^2+\lambda^{-2} -2) \left(\tau_0 + 2\sigma_0\zeta - \tau_0\zeta^2 \right)^2 -4(\zeta^2+1)^2
  \] respectively. 
  Moreover, for all $\zeta \in GF(q)$ it holds that  $\theta_{\alpha,\omega}(\zeta)$ is a non-zero square in $GF(q)$ if and only if $\theta_{0,\lambda}(\zeta)$ is a non-zero square in $GF(q)$.
  Thus $f(\zeta) := \theta_{\alpha,\omega}(\zeta) \theta_{0,\lambda}(\zeta)$ is a square in $GF(q)$ for all $\zeta \in GF(q)$.

  We now show that $f$ is not the square of a  polynomial in $\overline{GF(q)}[\zeta]$ and hence show that $y^2-f(x)\in GF(q)[x,y]$ is absolutely irreducible.
  Observe that, $\theta_{0,\lambda}$ is the difference of two squares and hence  has a factorization
  \begin{align*}
    \theta_{0,\lambda}(\zeta)
    =   &\left( ( (\lambda - \lambda^{-1})\tau_0 -2) + 2(\lambda-\lambda^{-1})\sigma_0\zeta - ( (\lambda-\lambda^{-1})\tau_0 + 2)\zeta^2\right)\\
        &\left( ( (\lambda - \lambda^{-1})\tau_0 + 2) + 2(\lambda-\lambda^{-1})\sigma_0\zeta - ( (\lambda-\lambda^{-1})\tau_0 - 2)\zeta^2\right).
  \end{align*}
  Moreover, observe that if $\theta_{0,\lambda}$ has a repeated root $\zeta_0$ then $\zeta_0^2 = -1$.
  However, one can easily check that no such root can exist, whereby $\theta_{0,\lambda}$ is separable.
  
  If $f$ is the square of some polynomial in $\overline{GF(q)}[\zeta]$, then all of its linear factors must appear with even multiplicity.
  However, $\deg{\theta_{0,\lambda}}=\deg{\theta_{\alpha,\omega}}$ and $\theta_{0,\lambda}$ is separable.
  Hence $f$ is the square of a polynomial if and only if $\theta_{\alpha,\omega} = \kappa \theta_{0,\lambda}$ for some $\kappa \in GF(q)$.
  Notice that $(\tau_0 + 2\sigma_0\zeta - \tau_0\zeta^2)$ and $1+\zeta^2$ are both irreducible in $GF(q)[\zeta]$ and as  such there does not exist $\mu,\eta \in GF(q)$ such that
  \[
    \mu(\tau_0 + 2\sigma_0\zeta - \tau_0\zeta^2)^2 + \eta (1+\zeta^2)^2 = (\tau_0 + 2\sigma_0\zeta - \tau_0\zeta^2)(\tau_0 + 2\sigma_0\zeta - \tau_0\zeta^2).
  \]
  So, as $\theta_{\alpha,\omega} = k\theta_{0,\lambda}$, we must have
    $4\alpha(\alpha^2 - 1)(\omega-\omega^{-1})^2 =0$ and as  $\omega \neq \pm 1$, we must have either $\alpha^2 =1$ or $\alpha=0$.
    We consider each case in turn.
  If $\alpha^2 = 1$, then by examining the $(1+\zeta^2)^2$ coefficients of $\theta_{0,\lambda}$ and $\theta_{\alpha,\omega}$ we see that we must have
  $\kappa = (\omega + \omega^{-1})^2$.
  Then, from the $(\tau_0 + 2\sigma_0\zeta - \tau_0\zeta^2)^2$ coefficients, we obtain
  \[
    -4(\omega^2 +\omega^{-2} -2) = (\omega^2 + \omega^{-2} + 2)(\lambda^2 + \lambda^{-2} - 2).
  \] 
  Now, by re-arranging we see that $\omega$ must be a root of the polynomial
\begin{equation}\label{eq3}
    (\lambda^2 + \lambda^{-2} +2)\omega^4 + 2(\lambda^2+\lambda^{-2} -6)\omega^2  + (\lambda^2+\lambda^{-2}+2) = 0.
\end{equation}
  Clearly, if $\omega_0$ is a root then $\omega_0^2$ is a root of (\ref{eq3}) viewed as a quadratic in $\omega^2$.
  However, the discriminant of (\ref{eq3}) as a quadratic equation in $\omega^2$  is given by  $-2^6(\lambda -\lambda^{-1})^2$ which is a non-zero non-square in $GF(q)$.
  Thus, (\ref{eq3}) has no solutions and we can conclude that $\alpha^2 = 1$ is not a possible case.

  This leaves only the case where $\alpha=0$.
  By examining constant terms of $\theta_{0,\lambda}$ and $\theta_{\alpha,\omega}$ it follows that $\kappa = 1$.
  Hence, $\omega - \omega^{-1} = \lambda - \lambda^{-1}$.
  There are then four possibilities for $\omega$. They are $\pm \lambda$ or $\pm \lambda^{-1}$.
  Thus $(\alpha,\omega) = (0, \pm \lambda)$ or $(\alpha,\omega) = (0,\pm \lambda^{-1})$.
  However, we know this is not the case. 
  We can then conclude that $f$  is not the square of a polynomial $\overline{GF(q)}[\zeta]$.
  Therefore, we have shown that $y^2-f(x) \in GF(q)[x,y]$ is absolutely irreducible. 

  We can now finish the proof.
  Define $N = |\{ (x,y) \in GF(q)^2 \mid y^2 = f(x) \}|$ and observe that as $\deg{f} =8$, we get $N \geq 2q-8$.
    We can then apply  Theorem~\ref{Solutions}  which yields that $|N-q| \leq 7\sqrt{q}$,
    whereby, $(q-1)(q-64) \leq 0$ and hence $q \leq 64$.
    Thus,  as we have assumed that $q > 64$ we get a contradiction.
    This proves the lemma. 
  \end{proof}

  We require one more lemma before we can finish our argument.
  \begin{lemma}\label{lem:faithful2}
    We have that $K$ acts faithfully on $\Delta_3(t)$.
    \end{lemma}
    \begin{proof}
      Let  $g \in K$ fix  $\Delta_3(t)$ vertex-wise and suppose that there exists  $z \in \Delta_2(t)$ such that $z^g \neq z$.
      If $\Delta_1(z) \cap \Delta_3(t) \neq \emptyset$ then as $\Delta_1(z)\cap \Delta_1(t) \neq \emptyset$, the graph $\mathcal{C}(L,X)$ contains a 4-cycle.
      This contradicts Lemma~\ref{lem:4cycle} and thus we may suppose that $\Delta_1(z) \cap \Delta_3(t) = \emptyset$.

      Choose any $v \in \Delta_3(t)$ and observe that because $\mathcal{C}(L,X)$ has diameter 3, $2 \leq d(z,v) \leq 3$.
      If $d(z,v) = 2$ then there exists $u \in \Delta_1(z) \cap \Delta_1(v)$ and as $v \in \Delta_1(u) \cap \Delta_3(t)$ we must have that $u^g = u$.
      Hence $\mathcal{C}(L,X)$ contains a 4-cycle on some vertex of $\Delta_1(z)\cap\Delta_1(t)$, $z$, $z^g$ and $u$. This is a contradiction,
      whereby, $d(z,v) = 3$.

      Let $u_1,u_2$ be such that $z,u_1,u_2,v$ is a path of length 3 from $z$ to $v$ in $\mathcal{C}(L,X)$.
      Note that, as $\Delta_1(u_2) \cap \Delta_3(t) \neq \emptyset$,  we have $u_2^g = u_2$. If $u_1 = u_1^g$, then we create a 4-cycle in $\mathcal{C}(L,X)$ and thus we must have $u_1 \neq u_1^g$.
      However, by Lemma~\ref{lem:4cycle}, we have that $\{u_2\} = \Delta_1(u_1) \cap \Delta_1(u_1^g)$.
      But as $u_1,u_1^g \in \Delta_2(t)$ and $g \in K$ we must have $\Delta_1(u_1) \cap \Delta_1(u_1^g) \subseteq \Delta_1(t)$.
      Hence, $u_2 \in \Delta_1(t)$ which contradicts the fact that $v \in \Delta_3(t)$.
      Therefore, for all $z \in \Delta_2(t)$ we have $\Delta_1(z) \cap \Delta_3(t) \neq \emptyset$ and thus $K$ acts faithfully on $\Delta_3(t)$.
    \end{proof}

    We can now prove the $q\equiv3\mod{4}$ part of Theorem~\ref{trm:1} when $q$ is not too small.
    \begin{lemma}\label{lem:2}
      If $q \geq 67$ holds, then $A  = P\Gamma L_2(q)$.
    \end{lemma}
    \begin{proof}
      Suppose that $q \geq 67$ and that  $K \normal A_t$  is the vertex-wise stabilizer of $\Delta_1(t)$ in $A_t$.
      By Lemma~\ref{lem:bound2}, we have that the $K$-orbits of $\Delta_3(t)$ all take the form $\{v, v^t\}$.
      Now we may proceed as in Lemma~\ref{lem:1}, using Lemma~\ref{lem:faithful2} in place of Lemma~\ref{lem:faithful1}, to deduce that $t \in \mathrm{Z}(A_t)$.
      Then it follows, as in Lemma~\ref{lem:1}, that $A  = P\Gamma L_2(q)$.
     \end{proof}

  Lemmas~\ref{lem:1} and ~\ref{lem:2} establish  Theorem~\ref{trm:1}(ii) provided $q \geq 73$. The cases when  $5< q < 73$ can be settled using the following \textsc{Magma} code.
  This program takes roughly 30 seconds to run, and so the online version suffices!
\begin{verbatim}
Q := [ k : k in [ 7.. 71] | (k mod 2 eq 1) and #Factorisation(k) eq 1];
for q in Q do
L := PSL(2,q);
X := {@ x : x in L | Order(x) eq 2 @};
Edges := {{i,j} : i,j in [ 1 .. #X] | i lt j and Order(X[i]*X[j]) eq 2};
Gam := Graph<#X | Edges>;
A := AutomorphismGroup(Gam);
print q, 1 eq #A/#PGammaL(2,q);
end for; 
\end{verbatim}
  
  Therefore the proof of Theorem~\ref{trm:1} is complete.

\end{document}